\documentclass[12pt,reqno,a4paper]{amsart}
\usepackage{
    amsmath,  amsfonts, amssymb,  amsthm,   amscd,
    gensymb,  graphicx, comment,  etoolbox, url,
    booktabs, stackrel, mathtools,enumitem, mathdots,  microtype, lmodern,    mathrsfs, graphicx, tikz,  longtable,tabularx, float, tikz, pst-node, tikz-cd, multirow, tabularx, amscd,  bm, array, makecell, diagbox, booktabs,ragged2e, caption, subcaption }
\usepackage{makecell,slashbox}
\usepackage{xcolor}
\usepackage[utf8]{inputenc}
\usepackage{microtype, fullpage, wrapfig,textcomp,mathrsfs,csquotes,fbb}
\usepackage[colorlinks=true, linkcolor=blue, citecolor=blue, urlcolor=blue, breaklinks=true]{hyperref}
\usepackage[capitalise]{cleveref}
\usepackage{todonotes}
\usetikzlibrary{positioning}
\usetikzlibrary{shapes,arrows.meta,calc}
\usetikzlibrary{arrows}

\newtheorem{theorem}{Theorem}[section]
\newtheorem*{theorem*}{Theorem}

\newtheorem{corollary}[theorem] {Corollary}
\newtheorem{definition}[theorem]{Definition}
\newtheorem{example}[theorem]{Example}
\newtheorem{lemma}[theorem]{Lemma}

\newtheorem{proposition}[theorem]{Proposition}
\newtheorem{remark}[theorem]{Remark}

\newcommand\R{\mathbb{R}}
\newcommand\Z{\mathbb{Z}}
\newcommand\C{\mathbb{C}}

\newcommand{\TC}{\mathrm{TC}}

\newcommand{\ct}{\mathrm{cat}}

\newcommand{\sct}{\mathrm{secat}}
\newcommand{\kr}{\mathrm{ker}}
\newcommand{\cd}{\mathrm{cd}}

\newcolumntype{x}[1]{>{\centering\arraybackslash}p{#1}}

\begin{document}
\title[]{Sequential parametrized topological complexity of group epimorphisms}

\author[N. Daundkar]{Navnath Daundkar}
\address{Department of Mathematics, Indian Institute of Science Education and Research Pune, India.}
\email{navnath.daundkar@acads.iiserpune.ac.in}

\thanks{}

\begin{abstract} 
We introduce and study the sequential analogue of Grant's parametrized topological complexity of group epimorphisms, which generalizes the sequential topological complexity  of groups.
We derive bounds for sequential parametrized topological complexity based on the cohomological dimension of certain subgroups, thereby extending the corresponding bounds for sequential topological complexity of groups. We also obtain sequential analogs of (new) lower bounds on parametrized topological complexity of epimorphisms which are recently obtained by Espinosa Baro, Farber, Mescher and Oprea. 
Finally, we utilize these results to provide alternative computations for the sequential parametrized topological complexity of planar Fadell–Neuwirth fibrations.
\end{abstract}
\maketitle

\section{Introduction}\label{sec:intro}
The \emph{topological complexity} of a space $X$, denoted by $\TC(X)$ is defined as the least positive integer $r$ for which $X\times X$ can be covered by open sets $\{U_0,\dots, U_r\}$,  such that each $U_i$ admits a continuous local section of the free path space fibration
\[\pi:X^I\to X\times X ~~ \text{ defined by }~~\pi(\gamma)=(\gamma(0),\gamma(1)), \]
where $X^I$ denotes the free path space of $X$ with a compact open topology. 
Farber \cite{FarberTC} introduced the concept of topological complexity to analyze the difficulty of computing a motion planning algorithm for the configuration space $X$ of a mechanical system. Farber showed that $\TC(X)$ is a numerical homotopy invariant of $X$.
This invariant has been extensively studied over the past two decades.

The sequential analogue of topological complexity was introduced by Rudyak in \cite{RUD2010}, known as the \emph{sequential (or higher) topological complexity}. 
Consider the fibration of path connected space
\[
\pi_n: X^I\to X^n  \text{ defined by }  \pi_n(\gamma)=\bigg(\gamma(0), \gamma(\frac{1}{n-1}),\dots,\gamma(\frac{n-2}{n-1}),\gamma(1)\bigg).   \]
The \emph{sequential topological complexity} of $X$ is the least integer $r$ for which $X^n$ is covered by open sets $\{U_0,\dots, U_r\}$,  such that each $U_i$ admits a continuous local section of $\pi_n$.
Note that by definition, it follows that $\TC_2(X)=\TC(X)$ (see \cite{gonzalezhighertc}). 

The notion of parametrized topological complexity was introduced by Cohen, Farber, and Weinberger in \cite{PTC}, with the sequential analogue by Farber and Paul in \cite{SequentialPTC}. These parametrized motion planning algorithms offer greater universality and flexibility, functioning effectively in diverse scenarios involving external conditions. 
A parametrized motion planning algorithm takes a pair of configurations under identical external conditions and generates a continuous motion of the system that remains unchanged by those conditions.

We now briefly define the sequential parametrized topological complexity. For a fibration $p: E \to B$, 
let $E^n_B$ be the the space of all
$n$-tuples of points in $E$ all of which lie in a common fibre of $p$, and $E^I_B$ the space of all paths in $E$ with images in a single fibre. The fibration $\pi_n$ restricts to the subspace $E^I_B$ of $E^I$, which we denote by $\Pi_n: E^I_B \to E^n_B$. The \emph{sequential parametrized topological complexity} of $p : E \to B$, denoted by $\TC_n[p : E \to B]$, is the smallest integer $r$ such that there exists an open cover ${U_0, \dots, U_r}$ of $E^n_B$ where each $U_i$ admits a continuous section of $\Pi_n$. Note that $\TC_2[p : E \to B]$ is the \emph{parametrized topological complexity} defined by Cohen, Farber, and Weinberger.

There is an invariant called LS category, introduced by Lusternik and Schnirelmann in \cite{LScat}, which is closely related to sequential topological complexity. The LS category of a space $X$, denoted by $\ct(X)$, is the smallest integer $r$ such that there exist and open cover $\{U_0,\dots, U_r\}$ of $X$ where each inclusion $U\hookrightarrow X$ is nullhomotopic.
The authors of \cite{gonzalezhighertc} established the inequality
$$\mathrm{cat}(X^{k-1})\leq \TC_k(X)\leq \mathrm{cat}(X^k).$$
Both sequential topological complexity and LS-category are homotopy invariants, so for aspherical spaces, they depend only on the fundamental group. Specifically, we define $\ct(G) := \ct(K(G,1))$ and $\TC_n(G) := \TC_n(K(G,1))$. This leads to an interesting question of expressing the (sequential) topological complexity of an aspherical space in terms of its fundamental group. The well-known Eilenberg-Ganea theorem, presented in \cite{EilenbergGanea}, provides such a relation for  the LS-category, showing that for any group $G$,  the equality $\ct(G) = \mathrm{cd}(G)$ holds, where $\mathrm{cd}(G)$ is the cohomological dimension of $G$.

The study of the topological complexity of groups has attracted significant attention and has been approached from various perspectives. Alongside specific computations for various example classes, as seen in \cite{CohenPruidze}, \cite{GrantTCArtingropus}, \cite{Dranishnikovnonorsurf}, \cite{Li2022}, \cite{HughesLi2022}, \cite{Dranishnikov2020}, and \cite{TCKleinbottle}, broader methodologies have also been explored. For example, \cite{TCBredon} employs techniques from equivariant topology and Bredon cohomology to establish new bounds for the topological complexity of groups. In \cite{lbTCaspherical}, the authors derived various lower bounds on $\TC(G)$ involving the cohomological dimension of certain subgroups of $G$, while \cite[Proposition 3.7]{grantfibsymm} provides an upper bound for $\TC(G)$. 
Additionally, the authors of \cite{baro2023sequential} present lower bounds for the sequential topological complexity of groups.

Grant \cite{GrantPTC} introduced the notion of parametrized topological complexity for group epimorphisms. For an epimorphism $\alpha: G \to H$, the \emph{parametrized topological complexity of $\alpha$} is defined by $\TC[\alpha: G \to H] := \sct(\Delta: G \to G \times_H G)$, where $\Delta$ is the diagonal homomorphism given by $\Delta(g) = (g,g)$ (see \Cref{def: ptchomo} for details). This notion generalizes the topological complexity of groups (see \cite[Example 3.7]{GrantPTC}). Grant further adapts the parametrized version of results from \cite{lbTCaspherical} and \cite{grantfibsymm} (see \cite[Theorem 4.1 and Theorem 5.1]{GrantPTC}).

In this paper, we introduce the concept of sequential parametrized topological complexity for group epimorphisms, which generalizes the sequential topological complexity of groups (see \Cref{ex: tcnG}). For an epimorphism $\alpha: G \to H$, the \emph{sequential parametrized topological complexity of $\alpha$} is defined by
$$\TC_n[\alpha: G\to H]:=\sct(\Delta_n:G\to G^n_H),$$ 
where $\Delta_n$ is the diagonal homomorphism given by $\Delta_n(g) = (g, \dots, g)$ (see \Cref{def: ptchomo} for details). In \Cref{thm: lbtcnptc}, we establish a lower bound on $\TC_n[\alpha: G \to H]$ in terms of the cohomological dimension of a certain subgroup of $G^n_H$.
\begin{theorem*}
Let $\alpha: G\to H$ be a group epimorphism. Let $L$ be a subgroup of $G^n_H$ such that $L\cap N=\{1\}$ for any subgroup $N$ of $G^n$ with $gNg^{-1}=\Delta_n(G)$ for some $g\in G^n_H$. Then 
\[\TC_n[\alpha: G\to H]\geq \mathrm{cd}(L).\]
\end{theorem*}
\noindent The above theorem serves as a parametrized analogue of \cite[Theorem 2.1]{htcfarberaspherical} by Farber and Oprea. As a consequence, in \Cref{cor:seqanalogue-Grant-lb}, we establish the sequential analogue of Grant's result \cite[Theorem 4.1]{GrantPTC}. Recently, in \cite[Theorem 4.8]{baro2023sequential}, the authors provided a lower bound on $\TC[\alpha: G \to H]$ for type $F$ groups $G$ and $H$. We derive the corresponding sequential analogue in \Cref{thm: lb-subgp-incl}.\begin{theorem*}
Let $\alpha: G\to H$ be a group epimorphism of type F groups. Then 
\[\TC_n[\alpha:G\to H]\geq \mathrm{cd}(G^n_H)-k(\alpha),\] where 
\[k(\alpha)=\mathrm{max}\{\mathrm{cd}(\cap_{i=1}^{n-1} C(k_i))\mid k_i\in \kr(\alpha)~ \& ~ k_i\neq 1_G,~ 1\leq i\leq n-1\}\] and $C(k_i)$ is the centralizer of an element $k_i$ in $G$. 
\end{theorem*}

In \Cref{thm: sptcub}, we establish the sequential analogue of \cite[Theorem 5.1]{GrantPTC}, which provides an upper bound on $\TC[\alpha: G \to H]$.
\begin{theorem*}
Let $\alpha: G\to H$ be an epimorphism of groups. Suppose for a normal subgroup $N$ of $G$, we have  $[N,\kr(\alpha)]=\{1\}$. Then $\Delta_n(N)$ is a normal subgroup of $G^n_H$, and 
\[\TC_n[\alpha: G\to H]\leq \mathrm{cd}\bigg(\frac{G^n_H}{\Delta_n(N)}\bigg).\]    
\end{theorem*}

Fadell–Neuwirth fibrations are used as models for motion planning problems in Euclidean space involving
multiple agents with unknown initial position avoiding collisions with both each other and obstacles. The parametrized topological complexity of these fibrations is computed in \cite{PTC, PTCcolfree}, while the sequential parametrized topological complexity is computed in \cite{SequentialPTC}. In \cite[Theorem 6.2]{GrantPTC}, Grant provided an alternative computation of the parametrized topological complexity of planar Fadell–Neuwirth fibrations using the parametrized topological complexity of epimorphisms. We aim to derive the sequential analogue of \cite[Theorem 6.2]{GrantPTC}. For a planar Fadell–Neuwirth fibration (see \Cref{def:F-N -fibrations}), we emulate Grant's strategy to obtain the following sequential analogue of \cite[Theorem 6.2]{GrantPTC}:
\begin{theorem*}
Let $p:E\to B$ be a planar Fadell-Neuwirth fibration.
Then $$\TC_n[p:E\to B]=nt+s-2.$$
\end{theorem*}

\subsection{Structure of the paper}
In \Cref{sec: sptc}, we establish a relation between the sequential parametrized topological complexity of a fibration and its pullback (see \Cref{lem: sptc pullback}) and prove that fibre-preserving homotopy equivalent fibrations (\Cref{def: gen fib hteq}) have the same sequential parametrized topological complexity.
In \Cref{sec: sptcge}, we define the sequential parametrized topological complexity of a group epimorphism and show that it matches that of the corresponding realized fibration (see \Cref{thm: sptcc of realized fib}).
In \Cref{sec: bounds}, we derive a lower bound on the sequential parametrized topological complexity of a group epimorphism in terms of cohomological dimensions of certain subgroups (see \Cref{thm: lbtcnptc}) and an upper bound (see \Cref{thm: sptcub}). We also prove a sequential analogue (see \Cref{thm: lb-subgp-incl}) of the lower bound on parametrized topological complexity from \cite{baro2023sequential}.
Finally, in \Cref{sec: sptcFNfibrations}, we provide an alternate proof for the computations of the sequential parametrized topological complexity of planar Fadell–Neuwirth fibrations.

\section{Sequential parametrized topological complexity}\label{sec: sptc}
In this section, we define sequential parametrized topological complexity and examine its behavior under pullbacks. We then show that fibrations that are fibrewise homotopy equivalent have the same sequential parametrized topological complexity.
\begin{definition}
Given continuous maps $p_i:E_i\to B$, define the generalized fibre product, denoted by  $_{B}\prod_{i=1}^{n}E_i$ as 
\[_{B}\prod_{i=1}^{n}E_i:=\{(e_1,\dots,e_n)\in \prod_{i=1}^nE_i \mid p_i(e_s)=p_j(e_t) ~\text{for}~ 1\leq i\neq j\leq n,~ 1\leq s\neq t\leq n\}.\]    
\end{definition}

If $E_i=E$ and $p_i=p$ for all $1\leq i\leq n$, then we denote  $_{B}\prod_{i=1}^{n}E_i$ by $E^n_B$.
\begin{remark}\label{rmk: gen fibprod}
Observe that the fibre product  $(_{B}\prod_{i=1}^{n-1}E_i)\times_{B} E_n$ for $\beta_1=p_1\circ  \prod_{i=1}^{n-1}pr_i:(_{B}\prod_{i=1}^{n-1}E_i)\to B$ and $\alpha_n:E_n\to B$ coincides with the generalized fibred product $_{B}\prod_{i=1}^nE_i$. The map $pr_i$ is the projection onto the $i$-th factor. The following commutative diagram describes how the generalized fibre product is obtained.
\[\begin{tikzcd}
_{B}\prod_{i=1}^nE_i \arrow{r}{} \arrow[swap]{d}{} & E_n \arrow{d}{p_n} \\%
(_{B}\prod_{i=1}^{n-1}E_i) \arrow{r}{\beta_1}&B. 
\end{tikzcd}\]
\end{remark}

We will now recall the definitions of sectional category and sequential parametrized topological complexity.
Let $f: X\to Y$ be a map.
The \emph{sectional category} of $f$ is  denoted  by $\sct(f)$, is defined as the least integer $r$ for which  $Y$ has an open cover $\{W_0,\dots,W_r\}$ such that each open set admits a continuous homotopy section $s_i:W_i\to X$ for $0\leq i\leq r$.
If $f: X\to Y$ is a fibration, then $\sct(f)$ coincides with its \emph{Schwarz genus}, denoted $\mathrm{gen}(f)$ (refer to \cite{Svarc61}).
For example,  $\TC_n(X)=\mathrm{gen}(\pi_n)$ and $\ct(X)=\sct(\iota: \ast\xhookrightarrow{} X)$. 

For a fibration $p:E\to B$ consider a subspace of $E^I$ defined as
\[E^I_B:=\{\gamma: [0,1]\to E \mid p\circ \gamma(t)=b ~\text{for some}~ b\in B\}.\]
Define the map $\Pi_n: E^I_B\to E^n_B$ by 
\begin{equation}\label{eq: Pin}
\Pi_n(\gamma):= \bigg(\gamma(0), \gamma(\frac{1}{n-1}),\dots,\gamma(\frac{n-2}{n-1}),\gamma(1)\bigg).   
\end{equation}
It follows from the proposition proved in the appendix of \cite{PTC} that the map $\Pi_n$ is a fibration.
The fibration $\Pi_n$ is called the \emph{parametrised endpoint fibration}.

\begin{definition}[{\cite{SequentialPTC}}]
Let $p:E\to B$ be a fibration. Then
the $n$-th sequential parametrized topological complexity of $p$ is denoted by $\TC_n[p:E\to B]$, and defined as \[\TC_n[p:E\to B]:=\sct(\Pi_n).\]   
\end{definition}
 
Let $\Delta_n: E\to E^n_B$ be the diagonal map. We note that the $\sct(\Delta_n: E\to E^n_B)=\sct(\Pi_n)$. Therefore, we have $\TC_n[p:E\to B]=\sct(\Delta_n: E\to E^n_B)$.

The relationship between the parametrized topological complexity of a fibration and its pullback was established by Grant in \cite[Lemma 2.4]{GrantPTC}. The following result provides the sequential analogue.
\begin{lemma}\label{lem: sptc pullback}
Let $p:E\to B$ be a fibration and let $\tilde{p}: \tilde{E}\to \tilde{B}$ be a pullback fibration corresponding to a map $\tilde{B}\to B$.
Then \[\TC_n[\tilde{p}:\tilde{E}\to \tilde{B} ]\leq \TC_n[p:E\to B].\]
\end{lemma}
\begin{proof}
Consider the following diagram  of a pullback of fibrations
\[ \begin{tikzcd}
\tilde{E} \arrow{r}{f'} \arrow[swap]{d}{\tilde{p}} & E \arrow{d}{p} \\%
\tilde{B} \arrow{r}{f}&B.\\%
& 
\end{tikzcd}
\]
Note that $f'$ induces a map $\tilde{f'}: \tilde{E}^n_{\tilde{B}}\to E^n_B$.
Now observe that the pullback of $\Pi_n: E^I_B\to E^n_B$ along $\tilde{f'}$ is isomorphic as a fibration to $\Phi:\tilde{E}^I_{\tilde{B}}\to \tilde{E}^n_{\tilde{B}}$. Thus, \[\TC_n[\tilde{p}:\tilde{E}\to \tilde{B}] =  \sct(\Phi)\leq \sct(\Pi_n)=\TC_n[p: E\to B].\qedhere \]
\end{proof}

Before we proceed to establish one of the important aims of this section: the fibre-preserving homotopy invariance of the sequential parametrized topological complexity, we will recall some definitions from \cite{GrantPTC}.
\begin{definition}\label{def: fib hteq}\
 \begin{enumerate}
\item Let $p:E\to B$ and $q: X\to B$ be two fibrations. 
Then a fibrewise map from $p:E\to B$ to $q: X\to B$ is a map $f:E\to X$ such that $q\circ f=p$.

\item A fibrewise homotopy $F:E\times I\to X$ is a map such that $q(F(-,t))=p$ for all $t\in I$. Thus, $F$ is a homotopy between fibrewise maps $F(-,0)$ and $F(-,1)$.

\item Let $p:E\to B$ and $q: X\to B$ be two fibrations. Then $p$ and $q$ are called fibrewise homotopy equivalent, if there are fibrewise maps  $f: E\to X$ and $g:X\to E$ such that there are fibrewise homotopies  from $f\circ g$ to $Id_{X}$ and $g\circ f$ to $Id_{E}$, respectively.
\end{enumerate}   
\end{definition}

Suppose $f : E \to X$ is fibrewise map which is also a homotopy equivalence. Then it is well known that $f$ is a fibrewise homotopy equivalence.

It was shown by Farber and Paul \cite{SequentialPTC} that the sequential parametrized topological complexity of fibrewise homotopy equivalent fibrations over same base coincides.
Following the approach in \cite{GrantPTC}, we aim to generalize this result to fibrations over different base spaces.
We first recall some definitions and notation from \cite{GrantPTC}.
\begin{definition}\label{def: gen fib hteq}
 Let $p: E\to B$ and $q: X\to Y$ be two fibrations. 
\begin{enumerate}
\item A pair of maps $f:  B\to Y$ and $f':E\to X$ is called a fibre-preserving map from $p$ to $q$, if the following diagram commutes:
\[ \begin{tikzcd}
E \arrow{r}{f'} \arrow[swap]{d}{q} & X \arrow{d}{p} \\%
B \arrow{r}{f}&Y.
\end{tikzcd}
\] We denote such a pair by $(f',f):p\to q$.

\item A pair of homotopies $F':E\times I\to X $ and $F:B\times I\to Y$ is called a fibre-preserving homotopy if the following diagram is commutative:
\[ \begin{tikzcd}
E\times I \arrow{r}{F'} \arrow[swap]{d}{q} & X \arrow{d}{p} \\%
B\times I \arrow{r}{F}&Y.
\end{tikzcd}
\]
Denote such a pair by $(F',F)$. Note that $(F',F)$ is a fibre-preserving homotopy between $(f_0',f_0):p\to q$ and $(f_1',f_1):p\to q$.
\item  The fibrations $p$ and $q$ are said to be fibre-preserving homotopy equivalent if there exists two pairs of fibre-preserving maps, $(f',f): p\to q$ and $(g',g):q\to p$ such that there are a fibre-preserving homotopies from $(g'\circ f', g\circ g)$ to $(Id_E,Id_B)$  and from $(f'\circ g', f\circ g)$ to $(Id_X,Id_Y)$, respectively.
\end{enumerate}
\end{definition}

The fibre-preserving homotopy invairnace of the parametrized topological complexity of fibrations over different bases was establised by Grant in \cite[Proposition 2.6]{GrantPTC}. We will now prove the sequetial analogue of \cite[Proposition 2.6]{GrantPTC}.
\begin{proposition}
Let $p:E\to B$ and $q:X\to Y$ be fibrations. If there are homotopy equivalences $f'$ and $f$ such that the following diagram commutes
\[ \begin{tikzcd}
E \arrow{r}{f'} \arrow[swap]{d}{p} & X \arrow{d}{q} \\%
B \arrow{r}{f}&Y,
\end{tikzcd}\]
then $\TC_n[p:E\to B]=\TC_n[q:X\to Y]$.
\end{proposition}
\begin{proof}
The proof is essentially identical to that of \cite[Proposition 2.6]{GrantPTC}, with Grant's argument adapted to the sequential setting.
Since both $f':E\to X$ and $f: B\to Y$ are homotopy equivalences,  it follows from the proposition on page number $53$ of \cite{JPMay} that the pair $(f', f)$ is a fibre-preserving homotopy equivalence. Therefore, there exist a fibre-preserving map $(g',g):p\to q$ and fibre-preserving homotopies from $(g'\circ f', g\circ g)$ to $(Id_E,Id_B)$  and from $(f'\circ g', f\circ g)$ to $(Id_X,Id_Y)$, respectively.

Let $\Pi_n:E^I_B\to E^n_B$ and $\Pi_n':X^I_Y\to X^n_Y$ be fibrations defined as in \eqref{eq: Pin}. Observe that $(f',f)$ and $(g',g)$ induces fibre-preserving maps $(\tilde{f'},\tilde{f}):\Pi_n\to \Pi_n'$  and $(\tilde{g'},\tilde{g}): \Pi_n'\to \Pi_n$, respectively.
More specifically we have the following commutative diagram:
\[ \begin{tikzcd}
E^I_B \arrow{r}{\tilde{f'}} \arrow[swap]{d}{\Pi_n} & X^I_Y \arrow{d}{\Pi'_n} \arrow{r}{\tilde{g'}} & E^I_B \arrow{d}{\Pi_n} \\%
E^n_B \arrow{r}{\tilde{f}}&X^n_Y \arrow{r}{\tilde{g} } & E^n_B,
\end{tikzcd}
\]
where the maps in the commutative diagram are defined as follows: $\tilde{f'}(\gamma)=f'(\gamma)$ for $\gamma\in E^I_B$, $\tilde{f}(e_1,\dots,e_n)=(f(e_1),\dots,f(e_n))$, $\tilde{g'}$ and $\tilde{g}$ are defined similarly. 
Our aim is to show $\tilde{f'}$ and $\tilde{f}$ are homotopy equivalences. Then our claim will follow using part-$(d)$ and part-$(e)$ of \cite[Lemma 2.1]{GrantPTC}. 

Let $(F',F)$ be a fibre-preserving homotopy from $(g'\circ f', g\circ g)$ to $(Id_E,Id_B)$. 
Define the fibre-preserving homotopy $(\tilde{F'},\tilde{F})$ by 
\[\tilde{F'}(\gamma,t)(s)=F'(\gamma(s),t) ~\text{and}~ \tilde{F}(e_1,\dots,e_n,t)=(F(e_1,t),\dots,F(e_n,t)).\]
Observe that $(\tilde{F'},\tilde{F})$ is a fibre-preserving homotopy from $(\tilde{g'}\circ \tilde{f'}, \tilde{g}\circ \tilde{f})$ to $(Id_{E^I_B}, Id_{E^n_B})$. 
Similarly, we can produce a fibre-preserving homotopy from  $(\tilde{f'}\circ \tilde{g'}, \tilde{f}\circ \tilde{g})$ to $(Id_{X^I_Y}, Id_{X^n_Y})$.
This implies $\tilde{f'}$ and $\tilde{f}$ are homotopy equivalences. This concludes the proof.
\end{proof}

\section{Sequential parametrized topological complexity of group epimorphisms}\label{sec: sptcge}
The parametrized topological complexity of group epimorphisms of discrete groups was introduced by Grant in \cite{GrantPTC}. In this section, we introduce a sequential analogue of  this concept.
We prove that the sequential parametrized topological complexity of group epimorphisms coincides with the  sequential parametrized topological complexity of the fibrations which realize them. For a group epimorhism, we show that a  lower bound on sequential parametrized topological complexity can be given as the sequential topological complexity of its kernel.

\begin{definition}[{\cite{GrantPTC}}]
Let $A$ and $B$ be pointed spaces and $\phi:A\to B$ be a pointed map. Then we say that $\phi$ realizes the homomorphism $\alpha: H'\to H$ of groups, if both $A$ and $B$ are $K(H',1)$ and $K(H,1)$ spaces, respectively and $\phi_{\ast}=\alpha:\pi_1(A)\to \pi_1(B)$.
\end{definition}

The following result enables the definition of the sectional category of group homomorphisms.

\begin{theorem}[{\cite[Lemma 3.2]{GrantPTC}}]
Let $H$, $H'$ be two discrete groups and let $\alpha: H'\to H$ be a homomorphism of groups. Then $\alpha$ can be realized by a pointed map. Suppose $\alpha$ is realized by two such maps, then their sectional category coincides.   
\end{theorem}
 
\begin{definition}[{\cite{GrantPTC}}\label{def:sct-homomorphisms}]
Let $f$ be any map which realizes a group homomorphism $\alpha : G \to H$. The sectional category of $\alpha$ is defined as $\sct(\alpha):=\sct(f)$.    
\end{definition}  
\noindent Let $j:\{1_G\}\hookrightarrow G$ be the inclusion homomorphism and  $\Delta_n: G\to G^n$ be the diagonal homomorphism. Then one can observe that $\ct(G)= \sct(j :\{1_G\}\hookrightarrow G)$ and  $\TC_n(G) = \sct(\Delta_n : G \to G^n)$.

Let $\alpha_i:G_i\to H$ be group homomorphisms for $1\leq i\leq n$. We define their \emph{generalized fibred product}, denoted by $_{H}\prod_{i=1}^{n}G_i$ as \[_{H}\prod_{i=1}^{n}G_i:=\{(g_1,\dots,g_n)\in G^n \mid \alpha_i(g_s)=\alpha_j(g_t) ~\text{for}~ 1\leq i\neq j\leq n,~ 1\leq s\neq t\leq n\}.\]
If $G_i=G$ and $\alpha_i=\alpha$ for all $1\leq i\leq n$, then we denote  $_{B}\prod_{i=1}^{n}G_i$ by $G^n_H$.
\begin{definition}\label{def: ptchomo}
Consider a group homomorphism $\alpha:G\to H$. Then the sequential parametrized topological complexity of $\alpha$  is denoted by $\TC_n[\alpha:G\to H]$ and defined as \[\TC_n[\alpha:G\to H]:=\sct(\Delta_n:G\to G^n_H).\]
\end{definition}

A fibration $p: E \to B$ is said to be \emph{$0$-connected} if it is surjective and has path-connected
fibres. It is known that only epimorphisms can be realized by $0$-connected fibrations. Therefore, we restrict ourselves to this setting.
Grant \cite[Proposition 3.5]{GrantPTC} showed that the parametrized topological complexity of group epimorphisms coincides with the   parametrized topological complexity of fibrations which realizes them. Here we prove sequential analogue of this result.
\begin{proposition}\label{thm: sptcc of realized fib}
Let $p: E\to B$ be a $0$-connected fibration which realizes the epimorphism $\alpha:G\to H$. Then \[\TC_n[p:E\to B]=\TC_n[\alpha: G\to H].\]
\end{proposition}
\begin{proof}
 Consider the following diagram 
\[\begin{tikzcd}
E \arrow{dr}{d} \arrow{rr}{h}
& & E^I_B \arrow{dl}{\Pi_n} \\
& E^n_B, 
\end{tikzcd}\] 
where $h$ is defined by sending a point $e$ to the constant path at $e$,  $d$ is the diagonal map and $\Pi_n$ defined in \eqref{eq: Pin}.
Note that $h$ is a homotopy equivalence, since it has a homotopy inverse $h':E^I_B\to E$ defined by sending a path to its midpoint.  
Since $h$ is a homotopy equivalence, the equality $\TC_n[p: E\to B]=\sct[d:E\to E^n_B]$ follows from \cite[Corollary 3.7]{AS}. Thus, it is enough to show that the map $d:E\to E^n_B$ realizes the diagonal map $\Delta_n: G\to G^n_H$.

We now consider the Mayer-Vietoris sequence  corresponding to the fibre sequence $(\Omega B)^{n-1}\to E^n_B\to E^n$ (see \cite[Corollary 2.2.3]{MP}) :
\[\begin{tikzcd}
\dots \arrow[r] & \pi_k(E^n_B) \arrow[r]
& \pi_k(E)^n \arrow[r]
\arrow[d, phantom, ""{coordinate, name=Z}]
& \pi_k(B)^{n-1} \arrow[dll, dashed,
"",
rounded corners,
to path={ -- ([xshift=2ex]\tikztostart.east)
|- (Z) [near end]\tikztonodes
-| ([xshift=-2ex]\tikztotarget.west)
-- (\tikztotarget)}] \\
 & \pi_2(E^n_B) \arrow[r]
&\pi_2(E)^n \arrow[r]
\arrow[d, phantom, ""{coordinate, name=Z}]
& \pi_2(B)^{n-1} \arrow[dll,
"",
rounded corners,
to path={ -- ([xshift=2ex]\tikztostart.east)
|- (Z) [near end]\tikztonodes
-| ([xshift=-2ex]\tikztotarget.west)
-- (\tikztotarget)}] \\
 & \pi_1(E^n_B) \arrow[r]
& \pi_1(E)^n \arrow[r]
\arrow[d, phantom, ""{coordinate, name=Z}]
& \pi_1(B)^{n-1} \arrow[dll,
"",
rounded corners,
to path={ -- ([xshift=2ex]\tikztostart.east)
|- (Z) [near end]\tikztonodes
-| ([xshift=-2ex]\tikztotarget.west)
-- (\tikztotarget)}] \\
& \pi_0(E^n_B) \arrow[r]
& \pi_0(E)^n \arrow[r]
& \pi_0(B)^{n-1},
\end{tikzcd}\]
where the map  $\pi_1(E)^n\to \pi_1(B)^{n-1}$ of pointed sets is defined by \[(g_1,\dots, g_n)\to (\alpha(g_1)\alpha(g_n)^{-1},\dots, \alpha(g_{n-1})\alpha(g_n)^{-1}).\]
Then again using \cite[Corollary 2.2.3]{MP}, we obtain the isomorphism $\pi_1(E^n_B)\cong \pi_1(E)^n_{\pi_1(B)}=G^n_H$. 
Also it follows that $E^n_B$ is $K(G^n_H,1)$. Additionally, since $B$ is $K(H,1)$ space, we have the following short exact sequence: 
\[\begin{tikzcd}
    1=\pi_2(B)^{n-1} \arrow{r}{}& \pi_1(E^n_B)\arrow{r}{}& \pi_1(E)^n\arrow{r}{}&\pi_1(B)^n \arrow{r}{}&\dots
\end{tikzcd}\]
This implies the homomorphism $\pi_1(E^n_B)\to \pi_1(E)^n$ induced by the inclusion $E^n_B\to E^n$ is injective. Therefore, the homomorphism $\pi_1(E^n_B)\to \pi_1(E)^n$ induced by the inclusion $E^n_B\to E^n$ is subgroup inclusion $G^n_H\hookrightarrow G^n$.
Now observing the commutative diagram
\[
\begin{tikzcd}
G \arrow{r}{d_*} \arrow[dr, "\Delta_n"']
& G^n_H \arrow[d, hook]\\
& G^n,
\end{tikzcd}
\]
it follows that the map $d:E\to E^n_B$ realizes the diagonal homomorphism $\Delta_n: G\to G^n_H$.
\end{proof}

\begin{example}\label{ex: tcnG}
 Let $\alpha: G\to \{1_G\}$ be the trivial epimorphism. Then $\TC_n[\alpha: G\to \{1_G\}]=\TC_n(G)$.   
\end{example}

\begin{lemma}\label{lem: tcnN}
Let $N=\mathrm{ker}(\alpha: G\to H)$. Then $\TC_n(N)\leq \TC_n[\alpha: G\to H]$.
\end{lemma}
\begin{proof}
Suppose the group homomorphism  $\alpha: G\to H$ is realized by a fibration $p:X\to Y$. Let $F$ be the fibre of $p$. Observe that $F$ is $K(N,1)$ space. Then from \cite{SequentialPTC} we obtain the desired inequality \[\TC_n(N)=\TC_n(F)\leq \TC_n[p:X\to Y]=\TC_n[\alpha:G\to H].\qedhere\]
\end{proof}

\section{Bounds}\label{sec: bounds}
In this section, we obtain various bounds on the sequential parametrized topological complexity of group epimorphisms. We begin by recalling some definitions.

\begin{definition}
An open set $V$ of $X$ is called $1$-categorical if for any cell-complex $Z$ of dimension less than or equal to $1$, each composition $Z \to V\hookrightarrow X$ is null-homotopic.     
\end{definition}
 
\begin{definition}[{\cite{Fox}}]
The least number of $1$-categorical subsets that covers a topological space $X$ is called the $1$-dimensional category.    
\end{definition}
It follows from \cite[Proposition 44]{Svarc61}
that, the $\ct_1(X)$ coincides with the sectional category of the universal cover $\tilde{X}\to X$. In particular, for aspherical space $X$ we have $\ct_1(X)=\ct(X)$.

In \cite[Theorem 2.1]{htcfarberaspherical},  Farber and Oprea show that a lower bound on the sequential $n$-th topological complexity of groups $G$ can be given in terms of the cohomological dimension of certain subgroups of $G^n$. We generalize this result in the context of sequential parametrized topological complexity of group epimorphisms.
\begin{theorem}\label{thm: lbtcnptc}
  Let $\alpha: G\to H$ be a group epimorphism.  Let $L$ be a subgroup of $G^n_H$ such that $L\cap N=\{1\}$ for any subgroup $N$ of $G^n$ with $gNg^{-1}=\Delta_n(G)$ for some $g\in G^n_H$. Then 
    \[\TC_n[\alpha: G\to H]\geq \mathrm{cd}(L).\]
\end{theorem}
\begin{proof}
Our argument closely follows the approach in \cite[Theorem 4.1]{GrantPTC}.
Consider a fibration $p:E\to B$ that realizes $\alpha:G\to H$. Let $X$ be a $K(L,1)$ space and $\psi: X\to E^n_B$ be a map  realizing the subgroup inclusion $L\hookrightarrow G^n_H$. Consider the following pullback diagram
\[ \begin{tikzcd}
Y \arrow{r}{} \arrow[swap]{d}{q} & E^I_B \arrow{d}{\Pi_n} \\%
X \arrow{r}{\psi}&E^n_B.
\end{tikzcd}
\]
It follows from \cite[Lemma 2.1(b)]{GrantPTC} that $\sct(q)\leq \TC_n[\alpha: G\to H]$.

Since $\mathrm{cd}(L)=\ct(X)$ by \cite{EilenbergGanea} and by asphericity $\ct_1(X)=\ct(X)$, it suffices to show that $\ct_1(X)\leq \sct(q)$. 
Let $U$ be an open subset of $X$ with a continuous local section $s:U\to Y$ of $q$. We need to show that $U$ is $1$-categorical. It follows from \cite[Lemma 5.3]{lbTCaspherical} that an open set $U \subseteq X$ is $1$-categorical if and only if every composition $f: S^1 \to U \hookrightarrow X$ is null-homotopic.
Consider the following diagram 
\[ \begin{tikzcd}
\left[S^1,Y\right] \arrow{r}{} \arrow[swap]{d}{q_{\ast}} & \left[S^1,E^I_B\right] \arrow{d}{\Pi_{n\ast}} \arrow{r}{m_{\ast}}   & \left[S^1,E\right] \arrow{dl}{d_{\ast}}\\%
\left[S^1,X\right] \arrow{r}{\psi_{\ast}}& \left[S^1,E^n_B\right]. 
\end{tikzcd}
\]
There is a bijection between conjugacy classes in $\pi_1(X,x_0)$ and $[S^1,X]$.  Thus, each $[f]\in [S^1,X]$ is corresponds to a conjugacy class $[(a_1,\dots,a_n)]$ in $L$. Observe that, by the commutativity of the above diagram we get $\psi_{\ast}([f])=d_{\ast}([f'])$ for some $[f']\in [S^1,E]$. This implies that an element $(a_1,\dots,a_n)$ is conjugate to a some element of $\Delta_n(G)$ in $G^n_H$.
Therefore, $(b_1,\dots,b_n)(a_1,\dots,a_n)(b_1,\dots,b_n)^{-1}=(c,\dots,c)$ for some $(b_1\dots,b_n)\in G^n_H$.
But this implies that $(a_1,\dots,a_n)\in N$ as well. Since $L\cap N=\{1\}$, we get that $a_i=1$ for all $1\leq i\leq n$. Hence $f:S^1\to U\hookrightarrow X$ is nullhomotopic.
\end{proof}

\begin{remark}
Observe that, in \Cref{thm: lbtcnptc} we get \cite[Theorem 2.1]{htcfarberaspherical} as a  special case  when we consider a trivial epimorphism $\alpha:G\to \{1_G\}$.
\end{remark}

We now prove some special cases of \Cref{thm: lbtcnptc}.

\begin{corollary}\label{cor: cd-lb}
 Let $\alpha:G\to H$ be a group epimorphism. Let $A_i$ be subgroups of $G$ such that  $\cap_{i=1}^ng_iA_ig_i^{-1}=\{1\}$ for any $(g_1,\dots,g_n)\in G^n_H$. Then \[\TC_n[\alpha: G\to H]\geq \mathrm{cd}(_{H}\prod_{i=1}^nA_i).\]   
\end{corollary}
\begin{proof}
Let $L=_{H}\prod_{i=1}^nA_i$. Following \Cref{thm: lbtcnptc}, we need to show that $L\cap N=\{1\}$ for any subgroup $N$ of $G^n$ such that $gNg^{-1}=\Delta_n(G)$. Suppose $\bar{a}=(a_1\dots,a_n)\in L\cap N$ with the above property. Thus we get $g\bar{a}g^{-1}=(c,\dots,c)$, where $g=(g_1,\dots,g_n)\in G^n_H$.
This implies $g_ia_ig_i^{-1}\in \cap_{i=1}^ng_iA_ig_i^{-1}$. Since $\cap_{i=1}^ng_iA_ig_i^{-1}=\{1\}$, we have $a_i=1$ for $1\leq i\leq n$. This concludes the proof.
\end{proof}

The following result is a generalization of Grant's result in \cite[Theorem 4.1]{GrantPTC}.
\begin{corollary}\label{cor:seqanalogue-Grant-lb}
Let $\alpha:G\to H$ be group epimorphism. Let $A_i$ be subgroups of $G$ for $1\leq i
\leq n$ such that $gA_ig^{-1}\cap A_j=\{1\}$ for $i\neq j$ and $g\in \mathrm{ker}(\alpha)$.  Then \[\TC_n[\alpha: G\to H]\geq \mathrm{cd}(_{H}\prod_{i=1}^nA_i).\]    
\end{corollary}
\begin{proof}
Let $L=_{H}\prod_{i=1}^nA_i$.
Suppose $\bar{a}=(a_1\dots,a_n)\in L\cap N$ with the property, $hNh^{-1}=\Delta_n(G)$ for some $h=(h_1,\dots,h_n)\in G^n_H$. Thus we get $h\bar{a}h^{-1}=(c,\dots,c)$, where $h=(h_1,\dots,h_n)\in G^n_H$. This implies $h_ia_ih_i^{-1}=h_ja_jh_j^{-1}$ for any $1\leq i,j\leq n$. From this we get $(h_j^{-1}h_i)a_i(h_j^{-1}h_i)^{-1}=a_j$. 
Note that $h_j^{-1}h_i\in \mathrm{ker}(\alpha)$. Thus from hypothesis we get that  $(h_j^{-1}h_i)a_i(h_j^{-1}h_i)^{-1}=1$. This gives $a_i=1$ for $1\leq i\leq n$. This proves $L\cap N=\{1\}$ for any subgroup $N$ of $G^n$ such that $gNg^{-1}=\Delta_n(G)$. Then the conclusion follows from \Cref{thm: lbtcnptc}.
\end{proof}

We now give another description of fibre product $G^n_H$ of groups which is a straightforward generalization of \cite[Lemma 4.7]{baro2023sequential}.

\begin{lemma}\label{lem: alt-decr-fib-prod}
    Let $\alpha: G\to H$ be a group epimorphism of type F groups. Then \[G^n_H=((\mathrm{ker}\ \alpha )^{n-1} \times \{1\})\cdot \Delta_n(G).\]
\end{lemma}
\begin{proof}
Let $(k_1g,k_2g,\dots,k_{n-1}g,g)\in ((\mathrm{ker}\  \alpha)^{n-1}\times \{1\})\cdot \Delta_n(G)$. Since $\alpha(k_ig)=\alpha(g)$, we get that $((\mathrm{ker} \ \alpha)^{n-1}\times \{1\})\cdot \Delta_n(G)\subseteq G^n_H$.
To show the other inclusion, consider $(g_1,\dots,g_n)\in G^n_H$. One can write $(g_1,\dots,g_n)=(g_1g_n^{-1},\dots, g_{n-1}g_n^{-1},1)\cdot (g_n,\dots,g_n)$, Since $g_ig_n^{-1}\in \mathrm{ker} \ \alpha$, we have $G^n_H\subseteq((\mathrm{ker} \ \alpha)^{n-1} \times \{1\})\cdot \Delta_n(G)$.
\end{proof}

Now, we recall the definition of type $F$ groups.
\begin{definition}\label{def:typeF}
A group $H$ is said to be of type $F$ if there is a finite CW-complex of type $K(H,1)$.    
\end{definition}
A lower bound for the sectional categories of subgroup inclusions has been obtained in \cite[Theorem 3.10]{baro2023sequential}.
One can observe that for a group epimorphism $\alpha:G \to H$, the homomorphism $\Delta_n: G\to G^n_H$ can be thought of as a subgroup inclusion. 
We now prove a sequential analogue of \cite[Theorem 4.8]{baro2023sequential}.

\begin{theorem}\label{thm: lb-subgp-incl}
Let $\alpha: G\to H$ be a group epimorphism of type F groups. Then 
\[\TC_n[\alpha:G\to H]\geq \mathrm{cd}(G^n_H)-k(\alpha),\] where 
\[k(\alpha)=\mathrm{max}\{\mathrm{cd}(\cap_{i=1}^{n-1} C(k_i))\mid k_i\in \kr(\alpha)~ \& ~ k_i\neq 1_G,~ 1\leq i\leq n-1\}\] and $C(k_i)$ is the centralizer of an element $k_i$ in $G$. 
\end{theorem}
\begin{proof}
Our argument closely follows that of \cite[Theorem 4.8]{baro2023sequential}.
We note the following inequality which follows from  \cite[Theorem 3.10]{baro2023sequential}: \[\TC_n[\alpha:G\to H]\geq \mathrm{cd}(G^n_H)-\kappa_{G^n_H,\Delta_n(G)},\] where     \[\kappa_{G^n_H,\Delta_n(G)}=\mathrm{max}\{\mathrm{cd}(\Delta_n(G)\cap g\Delta_n(G)g^{-1}) \mid g\in G^n_H\setminus \Delta_n(G)\}.\]
Therefore, it is enough to show that $k(\alpha)=\kappa_{G^n_H,G}$. 
From \Cref{lem: alt-decr-fib-prod}, we have 
\[\kappa_{G^n_H,\Delta_n(G)}=\mathrm{max}\{\mathrm{cd}(\Delta_n(G)\cap ((\bar{k},1)\cdot \tilde{h})\Delta_n(G)((\bar{k},1)\cdot \tilde{h})^{-1}) \mid \bar{k}\in \mathrm{ker}(\alpha)^{n-1}\setminus \bar{1}~ \&~ h\in G\},\]
where $\tilde{h}=\Delta_n(h)=(h,\dots,h)$ and $\bar{1}=(1,\dots,1)$ is the identity of $G^{n-1}$.
Since $\tilde{h}\Delta_n(G)\tilde{h}^{-1}=\Delta_n(G)$, we can rewrite the above description as 
\[\kappa_{G^n_H,\Delta_n(G)}=\mathrm{max}\{\mathrm{cd}(\Delta_n(G)\cap (\bar{k},1)\Delta_n(G)(\bar{k},1)^{-1}) \mid \bar{k}\in \mathrm{ker}(\alpha)^{n-1}\setminus \bar{1}\},\]
We now show that the group $\cap_{i=1}^{n-1} C(k_i)$ with  $ \bar{k}=(k_1,\dots,k_{n-1})\in \mathrm{ker}(\alpha)^{n-1}\setminus \bar{1}$ is isomorphic to $\Delta_n(G)\cap (\bar{k},1)\Delta_n(G)(\bar{k},1)^{-1}$.

Let $\tilde{g}=\Delta_n(g)$. Then note that $\tilde{g}\in (\bar{k},1)\Delta_n(G)(\bar{k},1)^{-1})$ if and only if $\tilde{g}=(\bar{k},1)\tilde{h}(\bar{k},1)^{-1}$ for some $\tilde{h}\in \Delta_n(G)$. 
This implies $g=k_ihk_i^{-1}$ and $g=h$. This implies $g=k_ihk_i^{-1}$. That is $k_ig=gk_i$ for $1\leq i\leq n-1$. This is equivalent to $g\in \cap_{i=1}^{n-1}C(k_i)$. This concludes $k(\alpha)=\kappa_{G^n_H,\Delta_n(G) }$.
\end{proof}

\begin{remark}
It would be interesting to know the relation between lower bounds obtained in \Cref{thm: lb-subgp-incl} and  \Cref{thm: lbtcnptc}. 
\end{remark}

For a group homomorphism $\alpha:G\to H$ of type $F$ groups, we describe $G^n_H$ as a semidirect product. This enables us to obtain the bounds on the cohomological dimension of $G^n_H$.

\begin{lemma}\label{lem: fib-prod-semidirect}
Let $\alpha:G\to H$ be a group epimorphism. Then \[\Psi:G^n_H\to (\kr(\alpha))^{n-1}\rtimes_{\phi} G ~\text{ defined by }~ \Psi(g_1,\dots,g_n)=(g_1g_n^{-1},\dots,g_{n-1}g_n^{-1},g_n)\] is a group isomorphism, where $\phi:G\to Aut((\kr(\alpha))^{n-1})$ is defined by \[\phi(g)(x_1,\dots,x_{n-1})=(gx_1g^{-1},\dots,gx_{n-1}g^{-1}).\]    
\end{lemma}
\begin{proof}
We first prove that $\Psi$ is injective. Suppose $\Psi(g_1,\dots,g_n)=\Psi(h_1,\dots,h_n)$. Then  $(g_1g_n^{-1},\dots,g_{n-1}g_n^{-1},g_n)=(h_1h_n^{-1},\dots,h_{n-1}h_n^{-1},h_n)$, giving $g_n=h_n$ and $g_ig_n^{-1}=h_ih_n^{-1}$, which implies $g_i=h_i$ for $1\leq i\leq n$.
Given any $(g_1,\dots,g_n)\in (\kr(\alpha))^{n-1}\rtimes_{\phi} G$ we have $\Psi(g_1g_n,\dots,g_{n-1}g_n, g_n)=(g_1,\dots,g_n)$, which implies $\Psi$ is surjective.
We now show that $\Psi$ is a homomorphism. 
\begin{align*}
\Psi((g_1,\dots,g_n)\cdot (h_1,\dots,h_n))&=
   \Psi(g_1h_1,\dots,g_nh_n) \\
   &=(g_1h_1h_n^{-1}g_n^{-1},\dots,g_{n-1}h_{n-1}h_n^{-1}g_n^{-1},g_nh_n)\\
   &=(g_1g_n^{-1}g_nh_1h_n^{-1}g_n^{-1},\dots,g_{n-1}g_n^{-1}g_nh_{n-1}h_n^{-1}g_n^{-1},g_nh_n)\\
   &= ((g_1g_n^{-1},\dots, g_{n-1}g_n^{-1})\phi(g_n)(h_1h_n^{-1},\dots,h_{n-1}h_n^{-1}), g_nh_n)\\
   &= (g_1g_n^{-1},\dots, g_{n-1}g_n^{-1},g_n)\bullet (h_1h_n^{-1},\dots,h_{n-1}h_n^{-1},h_n)\\
   &=\Psi(g_1,\dots,g_n)\bullet \Psi(h_1,\dots,h_n),
\end{align*}
where $\bullet$ is the multiplication in the semidirect product.
\end{proof}

We now prove the sequential analogues of \cite[Corollary 4.10, Corollary 4.11]{baro2023sequential}.
\begin{corollary}\label{cor:ub-cd-ker}
Let $\alpha:G\to H$ be a group epimorphism of type F groups. Then
\[\TC_n[\alpha:G\to H]\leq \mathrm{cd}(G)+\mathrm{cd}((\kr(\alpha))^{n-1}).\]
\end{corollary}
\begin{proof}
Since $\TC_n[\alpha:G\to H]=\sct(\Delta_n: G\to G^n_H)$, we have $\TC_n[\alpha:G\to H]\leq \ct(G^n_H)$. From \cite{EilenbergGanea}, for type F groups, we have $\ct(G^n_H)=\mathrm{cd}(G^n_H)$. This gives
\[\TC_n[\alpha:G\to H]\leq \mathrm{cd}(G^n_H).\]
It follows from \cite[VIII.2.4(b)]{B} that the cohomological dimension of the semidirect product of two groups is at most the sum of the cohomological dimensions of the factors. 
Then from \Cref{lem: fib-prod-semidirect} we have
$\mathrm{cd}(G^n_H)\leq \mathrm{cd}(G)+\mathrm{cd}((\kr(\alpha))^{n-1})$. This concludes the proof.
\end{proof}

\begin{corollary}
Let $\alpha:G\to H$ be a group epimorphism of type F groups with $\kr(\alpha))$ is also type F. Let $d=\mathrm{
cd}(\kr(\alpha))$ such that $H^d(\kr(\alpha); \Z[\kr(\alpha)])$ is $\Z$-free. Then 
\[n\cdot\cd(G)-(n-1)\cdot\cd(H)-k(\alpha)\leq\TC_n[\alpha: G\to H]\leq n\cdot\cd(G)-(n-1)\cdot\cd(H),\]
where \[k(\alpha)=\mathrm{max}\{\mathrm{cd}(\cap_{i=1}^{n-1} C(k_i))\mid (k_1,\dots,k_{n-1})\in (\mathrm{ker} \ \alpha)^{n-1}~ \& ~ k_i\neq 1 ~\text{ for }~ 1\leq i\leq n-1\}\] and $C(k_i)$ is a centralizer of an element $k_i$.
\end{corollary}
\begin{proof}
Since $\kr(\alpha)$ is of type $F$, we have a short exact sequence of type F groups
\[\begin{tikzcd}
1\arrow{r}& (\kr(\alpha))^{n-1} \arrow{r} & G^{n-1}\arrow{r}{\alpha^{n-1}}& H^{n-1} \arrow{r} &1.
\end{tikzcd}\]
 Thus, it follows from \cite[Theorem 5.5 (i)]{Bieri} and \cite[Corollary 2.5]{cdprod} that $\cd((\kr(\alpha))^{n-1})=(n-1)\cdot\cd(G)-(n-1)\cdot\cd(H)$.
 Therefore, the upper bound on $\TC_n[\alpha: G\to H]$ follows from \Cref{cor:ub-cd-ker}.
 We derive the lower bound using similar techniques. Observe that from \Cref{lem: fib-prod-semidirect} and again from \cite[Theorem 5.5 (i)]{Bieri}, we have $\cd(G^n_H)=\cd(G)+\cd((\kr(\alpha))^{n-1})$. 
 Now, we can use the expression $\cd((\kr(\alpha))^{n-1})=(n-1)\cdot\cd(G)-(n-1)\cdot\cd(H)$ and \Cref{thm: lb-subgp-incl} to get the desired lower bound on $\TC_n[\alpha: G\to H]$.
\end{proof}

The following result generalizes \cite[Theorem 5.1]{GrantPTC}.
\begin{theorem}\label{thm: sptcub}
Let $\alpha: G\to H$ be an epimorphism of groups. Suppose for a normal subgroup $N$ of $G$, we have  $[N,\kr(\alpha)]=\{1\}$. Then $\Delta_n(N)$ is a normal subgroup of $G^n_H$, and 
\[\TC_n[\alpha: G\to H]\leq \mathrm{cd}\bigg(\frac{G^n_H}{\Delta_n(N)}\bigg).\]
\end{theorem}
\begin{proof}
 We first show that $\Delta_n(N)$ is normal in $G^n_H$. Let $\bar{g}=(g_1,\dots,g_n)\in G^n_H$ and $\bar{ a}=(a,\dots,a)\in \Delta_n(I)$. Then $\bar{g}\bar{a}\bar{g}^{-1}=(g_1ag_1^{-1},\dots, g_nag_n^{-1})$.   
 Note that $g_j^{-1}g_i\in N$. Now consider $h^{-1}(g_j^{-1}g_i)h(g_j^{-1}g_i)^{-1}\in [I,\mathrm{ker} \ \alpha]=\{e\}$. 
 This gives  $g_iag_i^{-1}=g_jag_j^{-1}$ for all $1\leq i,j\leq n$. Therefore, $\bar{g}\bar{a}\bar{g}^{-1}\in \Delta_n(N)$. Thus $\Delta_n(N)$ is normal in $G^n_H$.

We denote the quotient $G^n_H/\Delta_n(N)$ by $W$.
Now consider the following diagram 
\[\begin{tikzcd}
    1 \arrow{r} & N \arrow{r}{\Delta_n|_{N}} \arrow[dr, hook] & G^n_H \arrow{r} & W \arrow{r} & 1\\%
    & & G \arrow{u}{\Delta_n}
\end{tikzcd}\]
and the corresponding diagram of aspherical spaces
\[\begin{tikzcd}
   & K(N,1) \arrow{r}{i} \arrow[dr] & K(G^n_H,1) \arrow{r}{p} & K(W,1)  & \\%
    & & K(G,1) \arrow{u}{f},
\end{tikzcd}\]
where the horizontal row is a fibration.  It follows from \cite[Lemma 2.2]{GrantPTC} that $\sct(f)\leq \ct(W)$. Thus, we obtain the desired inequality, as $\sct(f)=\TC_n[\alpha:G\to H]$ and $\ct(W)=\mathrm{cd}(W)$.
\end{proof}

The following result is a sequential analogue of \cite[Proposition 3.7]{grantfibsymm}.
\begin{corollary}
 Let $G$ be a torsion-free discrete group, and let $Z$ be its centre. Then   
 \[\TC_n(G)\leq \ct(G^n/\Delta_n(Z)).\]
\end{corollary}
\begin{proof}
 The proof follows by putting $H=e$ and $I=Z$ in \Cref{thm: sptcub}.   
\end{proof}

\begin{corollary}
If the extension 
\[\begin{tikzcd}
   e \arrow{r}& N\arrow{r} & G\arrow{r}{\alpha}& H \arrow{r} & e
\end{tikzcd}\] is central, then $\TC_n[\alpha:G\to H]=\mathrm{cd}(N^{n-1})$.
\end{corollary}
\begin{proof}
 Since $N$ is central, $[G,N]=\{e\}$. Therefore, $\Delta_n(G)$ is normal in $G^n_H$.
 Note that the map \[\frac{G^n_H}{\Delta_n(G)}\to N^{n-1}~ \text{defined by}~  [(g_1,\dots,g_n)]\to (g_1g_n^{-1},\dots, g_{n-1}g_n^{-1})\] is an isomorphism of groups.
Then from \Cref{thm: sptcub} we have \[\TC_n[\alpha:G\to H]\leq \mathrm{cd}(N^{n-1}).\]
It follows from \Cref{lem: tcnN} that $\TC_n(N)\leq \TC_n[\alpha:G\to H]$. Then we get the desired equality $\mathrm{cd}(N^{n-1})=\ct(N^{n-1})\leq \TC_n(N)$ using the fact $\mathrm{cd}(N)=\ct(N)$ from \cite{EilenbergGanea}.
\end{proof}

\section{Computations for planar Fadell–Neuwirth fibrations}\label{sec: sptcFNfibrations}
In this section, we give an alternate computation of the sequential parametrized topological complexity of planar Fadell-Neuwirth fibrations.

The ordered configuration space of $s$ points on $N$ is denoted by $F(N,s)$ and defined as
\[F(N,s)=\{(x_1,\dots,x_s)\in N^s \mid x_i\neq x_j ~\text{ for }~ i\neq j\}.\]

\begin{definition}[{\cite{FadellNeuwirth}}\label{def:F-N -fibrations}]
The maps \[p: F(N,s+t)\to F(N,s) ~\text{ defined by }~ p(x_1,\dots,x_{s+t})=(x_1,\dots,x_s)\] are called Fadell-Neuwirth fibrations.    
\end{definition}

The sequential parametrized topological complexity of these fibrations when $N=\R^2$ has been computed in \cite{SequentialPTC}. 
They showed the following.
\begin{theorem}[{\cite[Theorem 9.2]{SequentialPTC}}\label{thm: sptcFadNeu}]
Let $s \geq  2$, $t \geq 1$ and $n\geq 2$.
Then \[\TC_n[p:F(\R^{2},s+t)\to F(\R^{2},s)]=nt+s-2.\]
\end{theorem}
It can be observe that the configuration space $F(\R^2,k)$ can be thought of as a complement of the braid arrangement in $\C^k$. It is well known that such space is always aspherical. Therefore, Fadell-Neuwirth fibrations are realizations of corresponding induced homomorphism on the fundamental groups. 
This allows us to apply results in this paper to give an alternate proof of \Cref{thm: sptcFadNeu}. Recall that the bounds presented in this paper involves the cohomological dimensions of certain subgroups. Therefore, we now provide some well known tools to compute these algebraic invariants.

Bieri and Eckmann, as demonstrated in \cite{cdlemma}, have shown that if duality groups form a group extension, then the cohomological dimension of a middle group can be described in terms of the cohomological dimensions of the other two groups. Before we state this result, let us recall the definition of a duality group.

\begin{definition}
A group $G$ is called a duality group of dimension $n$ if there exists a $\Z G$-module $M$ and an element $x\in H_n(G; M)$ such that for any $\Z G$-module $B$, the map induced by cap product with $-\cap x$ is an  isomorphism $H^k(G;B)\to H_{n-k}(G;B\otimes M)$ for all $k$.
\end{definition}

The lemma provided proves to be valuable in computing the cohomological dimensions of groups.
\begin{lemma}[{\cite{cdlemma}}\label{lemms: cd}]\
\begin{enumerate}
\item Suppose $K$ is an $n$-dimensional duality group and $H$ is an $m$-dimensional group such that we have a following short exact sequence of groups
\[\begin{tikzcd}
1 \arrow{r}& K\arrow{r} &G\arrow{r}& H \arrow{r}&1. \end{tikzcd}\]
Then $G$ is an $n+m$ dimensional duality group.

\item Let $G$ be a non-trivial free group. Then $G$ is an $1$-dimensional duality group.
\end{enumerate}
\end{lemma}

Let $P_k = \pi_1(F (\R^2, k))$ represent the pure braid group on $k$ strands. Then, the map $p : F (\R^2, s + t) \to F (\R^2, s)$ realizes the epimorphism $\alpha: P_{s+t} \to P_s$, effectively removing the last $t$ strands.
Consider $p : F(\R^2, s + t) \to F(\R^2, s)$ as the Fadell-Neuwirth fibration with fiber $F(\R^2_{s, t},t)$, where $\R^2_{s,t} := \R^2\setminus \{1, 2,\dots, s\}$.
The kernel of $\alpha$ is $\bar{P}_{s,t} := \pi_1(F (\R^2_{s, t},t))$, known as the $t$-strand braid group of the $s$-th punctured plane. All these groups are iterated semi-direct products of free groups \cite{CohenSuciu}, and therefore are duality groups with a dimension equal to the number of free factors. In particular, $\mathrm{cd}(P_k) = k -1$ and $\mathrm{cd}(\bar{P}_{t,s}) =t$.

\begin{proof}[Proof of \Cref{thm: sptcFadNeu}]
We recall from \cite[Proposition 3.3]{lbTCaspherical} two subgroups $P_{s+t}$.
For $1\leq j\leq s+t-1$, consider a braid $\phi_j$ which pass the $j$-th strand over and behind the last $s+t-j$ strands before returning to their initial position.
Consider the subgroup $A$ generated by the braid $\phi_j$ for $1\leq j\leq s+t-1$. 
Note that $A$ is the free abelian group of rank $s+t-1$.
For any pure braid on $s+t-1$ strands we can create the braid on $s+t$ strands by fixing the last stand. Suppose $\psi:P_{s+t-1}\to P_{s+t}$ is this embedding. Define $B$ as the image of $\psi$.
(For more details we refer the reader to \cite{GrantPTC}).

Consider the following pullback diagram of an extension 
\[\begin{tikzcd}
1 \arrow{r}& \bar{P}_{t,s}\arrow{r} &P_{s+t}\arrow{r}& P_s \arrow{r}&1 \end{tikzcd}\] 
along the  map $\alpha_A\circ pr_A:A\times_{P_m}B\to P_s$.
\[\begin{tikzcd}
    1 \arrow{r} & \bar{P}_{t,s} \arrow{r}{} \arrow[dr, dotted] & P_{s+t} \arrow{r}{\alpha} & P_s \arrow{r} & 1\\%
    & & (A\times_{P_s}B)\times_{P_s} P_{s+t} \arrow{r} \arrow{u}{} & A\times_{P_m} B\arrow{r}  \arrow{u}{\alpha_A\circ pr_A}&1.
\end{tikzcd}\]
 Note that the dotted arrow exists because of the universal property of fibre products. 
 It was shown in \cite[Theorem 6.2]{GrantPTC} that $\mathrm{cd}(A\times_{P_s}B)=2t+s-2$. Therefore,
 using \Cref{lemms: cd} we get that 
 \[\mathrm{cd}((A\times_{P_s}B)\times_{P_s} P_{s+t})=t+2t+s-2=3t+s-2.\]

Recall that the iterated fibre product of $n-2$ many same maps $\alpha:P_{s+t}\to P_s$  is given by $P_{s+t}\times_{P_s}\dots\times_{P_s} P_{s+t}=_{P_s}P_{s+t}^{n-2}$.
Therefore, we can iterate the process of getting pullback extensions to get the following extension 
\[\begin{tikzcd}
1 \arrow{r}& \bar{P}_{t,s}\arrow{r} &(A\times_{P_s} B)\times_{P_s} (_{P_s}P_{s+t}^{n-2}) \arrow{r}& (A\times_{P_s} B)\times_{P_s} (_{P_s}P_{s+t}^{n-3})  \arrow{r}&1. 
\end{tikzcd}\] 
Thus by induction we get 
 \[\mathrm{cd}((A\times_{P_s} B)\times_{P_s} (_{P_s}P_{s+t}^{n-2}))=t+(n-1)t+s-2=nt+s-2.\]
One can observe that, if we denote $A_1=A$, $A_2=B$ and $A_i=P_{s+t}$ for $3\leq i\leq n$, then in our notation $(A\times_{P_s} B)\times_{P_s} (_{P_s}P_{s+t}^{n-2})= _{P_s}\prod_{i=1}^nA_i$.
Utilizing linking numbers, it is shown in \cite{lbTCaspherical} that $hAh^{-1}\cap B=\{1\}$ holds for all $h \in P_{s+t}$.
Consequently, one can check that the subgroups $A_i$ for $1\leq i\leq n$ satisfies the condition of \Cref{cor: cd-lb}. Thus, we obtain the  following inequality \[\TC_n[p : F (\R^2, s + t) \to F (\R^2, s)]\geq nt+s-2.\]

Now it remains to show that the upper bound also coincides with the $nt+s-2$.
The idea to get this is similar to above. 
Again using induction we can iteratively pullback an extension 
\[\begin{tikzcd}
1 \arrow{r}& \bar{P}_{t,s}\arrow{r} &P_{s+t}\arrow{r}& P_s \arrow{r}&1 \end{tikzcd}\] 
along $\alpha: P_{s+t}\to P_s$ to get the following extension
\[\begin{tikzcd}
1 \arrow{r}& \bar{P}_{t,s}\arrow{r} &_{P_s}P_{s+t}^{n}\arrow{r}& _{P_s}P_{s+t}^{n-1} \arrow{r}&1. \end{tikzcd}\]
Let $Z$ be the centre of $P_{s+t}$. Then taking quotient by $\Delta_n(Z)$ we get  the following extension
\[\begin{tikzcd}
1 \arrow{r}& \bar{P}_{t,s}\arrow{r} &_{P_s}P_{s+t}^{n}/\Delta_n(Z)\arrow{r}& _{P_s}P_{s+t}^{n-1}/\Delta_n(Z) \arrow{r}&1. \end{tikzcd}\]
Now we can use \Cref{lemms: cd} and induction to get 
\[\mathrm{cd}(_{P_s}P_{s+t}^{n}/\Delta_n(Z))=t+(n-1)t+s-2=nt+s-2.\]
Therefore, using \Cref{thm: sptcub} we get \[\TC_n[p : F (\R^2, s + t) \to F (\R^2, s)]\leq nt+s-2.\]
This completes the proof.
\end{proof}

\vspace{1cm}

\noindent\textbf{Acknowledgement:}
The author sincerely thanks the reviewer for the numerous insightful suggestions, comments, and feedback that significantly improved the exposition and presentation of this article.
The author expresses gratitude to Tejas Kalelkar for  insightful discussions and Rekha Santhanam for her feedback related to the exposition of this article. Author also thank Anurag Singh for his help in correcting Latex related issues. Additionally, the author acknowledges the support of National Board of Higher Mathematics (NBHM) through grant 0204/10/(16)/2023/R\&D-II/2789.

\bibliographystyle{plain} 
\bibliography{references}

\end{document}